\documentclass[letterpaper, 10 pt, conference]{ieeeconf}  

\IEEEoverridecommandlockouts                              

\usepackage[utf8]{inputenc}
\usepackage{amsmath}
\usepackage{amssymb}
\usepackage{amsfonts}
\usepackage{bm}

\let\labelindent\relax

\usepackage{amsthm}
\usepackage{overpic}
\usepackage{lipsum}
\usepackage{enumitem}
\usepackage{cases}
\usepackage{subcaption}
\usepackage{graphicx}
\usepackage{ulem}
\usepackage{xcolor}
\usepackage{cite}

\usepackage{algorithm, algorithmicx}
\usepackage{algpseudocode}
\usepackage{mathbbol}
\usepackage{comment}

\algnewcommand{\Inputs}[1]{%
  \State \textbf{Inputs:}
  \Statex \hspace*{\algorithmicindent}\parbox[t]{.8\linewidth}{\raggedright #1}
}
\algnewcommand{\Initialize}[1]{%
  \State \textbf{Initialize:}
  \Statex \hspace*{\algorithmicindent}\parbox[t]{.8\linewidth}{\raggedright #1}
}

{
    \theoremstyle{plain}
    
}

\DeclareCaptionLabelFormat{lc}{\MakeLowercase{#1}~#2}

\newtheorem{theorem}{Theorem}
\newtheorem{remark}{Remark}
\newtheorem{lemma}{Lemma}

\newtheorem{definition}{Definition}
\newtheorem{proposition}{Proposition}
\newtheorem{example}{Example}[]

\makeatletter
\def\maketag@@@#1{\hbox{\m@th\normalfont\normalsize#1}}
\makeatother

\title{\LARGE \bf
Collaborative Altruistic Safety in Coupled Multi-Agent Systems \vspace{-1.5ex}
}

\author{Brooks A. Butler$^1$, Xiao Tan$^2$, Aaron D. Ames$^2$, and Magnus Egerstedt$^3$
\thanks{This research was supported in part by an appointment to the Intelligence Community Postdoctoral Research Fellowship Program at the University of California, Irvine, administered by the Oak Ridge Institute for Science and Education (ORISE) through an interagency agreement between the U.S. Department of Energy and the Office of the Director of National Intelligence (ODNI). This work is also supported in part by TII under project \#A6847.}
\thanks{$^1$Brooks A. Butler is with the Department of Electrical Engineering and Computer Science at the University of California, Irvine (Email: {\tt\small bbutler2@uci.edu}).}
\thanks{$^2$Xiao Tan and Aaron D. Ames are with the Department of Mechanical and Civil Engineering, California Institute of Technology, Pasadena, CA 91125, USA (Email: {\tt\small xiaotan, ames@caltech.edu}).}
\thanks{$^3$Magnus Egerstedt is with the Department of Computer Science at the University of North Carolina, Chapel Hill, NC, 27599, USA (Email: {\tt\small magnus@unc.edu)}}
}

\begin{document}

\IEEEaftertitletext{\vspace{-1.5\baselineskip}}

\maketitle
\thispagestyle{empty}
\pagestyle{empty}

\begin{abstract}
This paper presents a novel framework for ensuring safety in dynamically coupled multi-agent systems through collaborative control. Drawing inspiration from ecological models of altruism, we develop collaborative control barrier functions that allow agents to cooperatively enforce individual safety constraints under coupling dynamics. We introduce an altruistic safety condition based on the so-called Hamilton’s rule, enabling agents to trade off their own safety to support higher-priority neighbors. By incorporating these conditions into a distributed optimization framework, we demonstrate increased feasibility and robustness in maintaining system-wide safety. The effectiveness of the proposed approach is illustrated through simulation in a simplified formation control scenario.
\end{abstract}

\section{Introduction}
Providing safety assurances for multi-agent systems is a challenging and widely applicable area of research, with relevant applications found in both multi-robot \cite{wang2017safety,lippi2020human} and cyber-physical systems \cite{zhang2021physical,chen2020safety,wang2023distributed}. One perspective of particular interest is when agents are considered distinct decision-makers with individual safety constraints, and agents must also consider how neighboring agent decisions affect their safety, and vice versa, collaboratively. 
This interplay between individual constraints and collective outcomes parallels behaviors found in natural systems, where cooperation and tradeoffs emerge despite individual interests.
We find examples of this phenomenon in ecology, where individual organisms (agents) may take altruistic actions to increase the likelihood of passing on shared genes, i.e., increasing their ``inclusive fitness," even at the cost of some individuals. 

The notion of inclusive fitness is encoded through Hamilton's Rule \cite{hamilton1963evolution,butler2025hamilton,karam2025resource}, which describes when an altruistic act is beneficial from the vantage point of genetic fitness, or the likelihood of certain genes to survive in a population of organisms.  Although genetic fitness might not be particularly relevant in multi-agent problems, the underpinning idea of trading off costs and benefits across team members relative to the safety-criticality of individual agents is meaningful. Therefore, in this paper, we construct a framework for facilitating altruistic safety guarantees in multi-agent systems using ecologically inspired techniques.

To discuss safety formally, we leverage the now well-established methods, control barrier functions (CBFs) \cite{ames2017tac}, developed for safety-critical control. A growing body of work extends these tools to networked multi-agent systems \cite{chen2020safety,xiao2025continuous,mestres2023distributed}. In many cases, safety-critical control for multi-agent safety is implemented using a global, or shared, safety objective, which can then be decomposed into individual safety constraints for each agent that may be solved in a distributed manner \cite{chen2020safety,jagtap2020compositional}. However, designing a centralized safety constraint may be challenging for systems where safety should be considered individually at the level of each agent. A common and relevant situation is when agents wish to avoid unsafe regions individually while subjected to coupling dynamic effects from neighbors (e.g., obstacle avoidance while trying to keep a formation) \cite{butler2024collaborative_ecc, butler2024collaborativesafety}.  

In contrast with past approaches, this paper makes two main contributions. First, it improves upon the method in \cite{butler2024collaborativesafety} for synthesizing collaborative control barrier function conditions for dynamically coupled multi-agent systems with locally defined safety constraints. By incorporating a virtual first-order safety controller for each agent, we show that the resulting coupled safety problem can be addressed using existing distributed optimization frameworks \cite{xiao2025continuous}. Second, it introduces an ecologically inspired framework for altruistic decision-making, defining an altruistic safety condition that accounts for the relative safety criticality of each agent. We show that incorporating this condition into the distributed optimization framework expands the set of feasible safe actions for agents with higher safety importance.

\section{Agent-Level Safety with Coupled Dynamics} \label{sec:coupled_safety_cond}
This section introduces notation for coupled multi-agent systems with agent-level safety constraints and shows how coupled CBF-based safety conditions may be synthesized for each agent by considering the high-order dynamics of the coupled system. The resulting safety filter can then be solved locally using an existing distributed optimization framework.   
\subsection{Coupled Dynamics} \label{sec:coupled_dyn}
For a given agent~$i \in [n]$, where $[n] = \{1, \dots, n\}$ denotes the set of agent indices, let $\mathcal{N}_i^+$ be the set of all neighbors $j \in [n]$ that have an interaction with the dynamics of node~$i$, where $\mathcal{N}_i^+ = \{j \in [n]: (i,j) \in \mathcal{E} \}$, where $\mathcal{E} \subseteq [n] \times [n]$.
Similarly, all nodes $j \in [n] $ whose dynamics are affected by node~$i$ are given by $\mathcal{N}_i^- = \{j \in [n]: (j,i) \in \mathcal{E} \}$,
with the complete set of neighboring nodes given by $\mathcal{N}_i = \mathcal{N}_i^+ \cup \mathcal{N}_i^-$.
Note that node~$i$ is included in both $\mathcal{N}_i^+$ and $\mathcal{N}_i^-$ as it both affects and is affected by itself. 
Further, we define the state vector for each node as $x_i \in \mathbb{R}^{N_i}$, with $N = \sum_{i \in [n]} N_i$ being the state dimension of the entire system, $N_i^+ = \sum_{j \in \mathcal{N}_i^+ \setminus \{ i\}} N_j$ the combined dimension of incoming neighbor states, and $x_{\mathcal{N}_i^+} \in \mathbb{R}^{N_i^+}$ denoting the combined state vector of all incoming neighbors.  Then, for each agent~$i \in [n]$, we can describe its state dynamics, which we assume are time-invariant and control-affine, as
\begin{equation} \label{eq:net_dyn_sys}
    \dot{x}_i = f_i(x_i, x_{\mathcal{N}_i^+}) + g_i(x_i) u_i,
\end{equation}
where $f_i:\mathbb{R}^{N_i + N_i^+} \rightarrow \mathbb{R}^{N_i}$ and $g_i:  \mathbb{R}^{N_i} \rightarrow \mathbb{R}^{N_i} \times \mathbb{R}^{M_i}$ are locally Lipschitz for all $i \in [n]$, and $u_i \in \mathcal{U}_i \subset \mathbb{R}^{M_i}$. 
For notational compactness, given a node~$i \in [n]$, we collect the 1-hop neighborhood state as $\mathbf{x}_i = (x_i, x_{\mathcal{N}_i})$ and the 2-hop neighborhood state as $\mathbf{x}_i^+ = (x_i, x_{\mathcal{N}_i}, x_{\mathcal{N}_j}:\forall j \in \mathcal{N}_i)$.

\subsection{Coupled Safety Condition Through High-Order Dynamics} \label{sec:1D_example}
Node-level safety constraints are defined by the set
\begin{equation}\label{eq:safe_set_i_fullstate}
    \begin{aligned}
        \mathcal{C}_i &= \left\{ x_i \in \mathbb{R}^{N_i} : h_i(x_i) \geq 0 \right\}, \forall i \in [n]
    \end{aligned}
\end{equation}
where $h_i:\mathbb{R}^{N_i} \rightarrow \mathbb{R}$ is a continuously differentiable function whose zero-super-level set defines the region which node $i\in [n]$ considers to be safe.

To investigate the effect of coupling dynamics  on individual safety,  let us suppose for now that node $i$ is governed by a smooth feedback control law $k_i(\mathbf{x}_i)$. More details on this control law are specified later (see Remark~\ref{rem:k_i}). Following the conventional high-order CBF design\cite{tan2021high}, a second control barrier function candidate is formulated by  
\begin{equation} \label{eq:h_i+}
    h_i^+(\mathbf{x}_i) = \mathcal{L}_{f_i}h_i(\mathbf{x}_i) + \mathcal{L}_{g_i}h_i(x_i)k_i(\mathbf{x}_i) + \alpha_i h_i(x_i),
\end{equation}
where $\mathcal{L}_{f_i}h_i(\mathbf{x}_i)$ and $\mathcal{L}_{g_i}h_i(x_i)$ denote the Lie derivatives of $h_i$ with respect to $f_i$ and $g_i$, respectively, and $\alpha_i \in \mathbb{R}_{>0}$, with the corresponding safety constraint set
\begin{equation}\label{eq:C_i+}
    \mathcal{C}_i^+ = \{ \mathbf{x}_i \in \mathbb{R}^{N_i + N_i^+}: h_i^+(\mathbf{x}_i) \geq 0 \}.
\end{equation}

The CBF candidate $h_i^+$ provides a means for synthesizing a coupled safety condition for node~$i$ that includes the inputs of its neighbors. We see this coupling explicitly in the computation of $\dot{h}_i^+$

\vspace{-2ex}
\footnotesize
\begin{align}
   \dot{h}_i^+ &= \sum_{j \in \mathcal{N}_i^+} \big[ \mathcal{L}_{f_j} \mathcal{L}_{f_i}  h_i(\mathbf{x}_i, \mathbf{x}_j) +  \mathcal{L}_{f_j} \left(\mathcal{L}_{g_i}  h_i(x_i) k_i(\mathbf{x}_i)\right) \big] \nonumber \\
    &
   \quad 
   +\sum_{j \in \mathcal{N}_i^+} \big[ \mathcal{L}_{g_j} \mathcal{L}_{f_i}  h_i(\mathbf{x}_i, x_j) +  \mathcal{L}_{g_j} \left( \mathcal{L}_{g_i}  h_i(x_i) k_i(\mathbf{x}_i) \right) \big] u_j \nonumber \\
    &
    \quad 
    + \alpha_i \left( \mathcal{L}_{f_i}  h_i(\mathbf{x}_i) + \mathcal{L}_{g_i}h_i(x_i) u_i \right), \label{eq:sec_der_h_i_control_aff}
\end{align}
\normalsize
which we use in the following definition and lemma.

\begin{definition}\cite[Collaborative Control Barrier Functions]{butler2024collaborativesafety}
    $h_i$ is a \textit{collaborative control barrier function} (CCBF) for node~$i \in [n]$ if there exists a feedback controller $k_i$, and $\forall \mathbf{x}_i \in \mathcal{C}_i^+$ there exists $(u_i, u_{\mathcal{N}_i^+}) \in \mathcal{U}_i \times \mathcal{U}_{\mathcal{N}_i^+}$ such that 
    \begin{equation} \label{eq:CCBF_cond}
        \dot{h}_i^+(\mathbf{x}_i^+, u_i, u_{\mathcal{N}_i^+}) + \beta_i h_i^+(\mathbf{x}_i)  \geq 0
    \end{equation}
    and
    \begin{equation}\label{eq:CCBF_cond_2}
        \begin{aligned}
             \mathcal{L}_{g_i}h_i(x_i) u_i +  \gamma_i h_i(x_i) \geq \mathcal{L}_{g_i}h_i(x_i) k_i(\mathbf{x}_i),
        \end{aligned}
    \end{equation}
    where $\alpha_i, \beta_i, \gamma_i \in \mathbb{R}_{\geq 0}$.
\end{definition}

\begin{lemma} \label{lem:CCBF}
    If $h_i$ is a CCBF, then, with  $(u_i, u_{\mathcal{N}_i^+})$  satisfying the conditions \eqref{eq:CCBF_cond} and \eqref{eq:CCBF_cond_2} applied, $\mathcal{C}_i$ is forward invariant. 
\end{lemma}
\begin{proof}
We have by \eqref{eq:CCBF_cond} that $\mathcal{C}_i^+$ is forward invariant. Thus, by \eqref{eq:h_i+}, we have that 
\begin{equation*}
    \mathcal{L}_{f_i}h_i(\mathbf{x}_i(t)) + \mathcal{L}_{g_i}h_ik_i(\mathbf{x}_i(t)) +  \alpha_i h_i(x_i(t)) \geq 0, \forall t\geq 0.
\end{equation*}
Furthermore, from \eqref{eq:CCBF_cond_2},  we know
\begin{equation*}
\begin{aligned}
    \mathcal{L}_{f_i}h_i(\mathbf{x}_i) & + \mathcal{L}_{g_i}h_i(x_i)u_i +  (\gamma_i + \alpha_i) h_i(x_i) \\
    & \geq \mathcal{L}_{f_i}h_i(\mathbf{x}_i) + \mathcal{L}_{g_i}h_i(x_i) k_i(\mathbf{x}_i) + \alpha_i h_i(x_i) \geq 0.
\end{aligned}
\end{equation*}
For any initial state $x_i(0)\in \mathcal{C}_i$, since $\dot{h}_i \geq - (\gamma_i + \alpha_i) h_i$, we establish that $h_i(x_i(t))\geq 0, \forall t\geq 0$ based on comparison lemma~\cite[Lemma~3.4]{Khalil2002}. Thus, $\mathcal{C}_i$ is forward invariant.
\end{proof}

\begin{remark}\label{rem:k_i}
The above analysis takes the feedback controller $k_i(\mathbf{x}_i)$ as given. In fact,  $k_i(\mathbf{x}_i)$ can be a design choice.
One may pick a $k_i(\mathbf{x}_i)$ such that the new safety constraint $h_i^+ (\mathbf{x}_i) \geq 0$ is easy or even trivial to satisfy, by, for example, choosing the ``half-Sontag" smooth safety filter from \cite{cohen2023characterizing} 
$$k_i(\mathbf{x}_i)=  \lambda(a(\mathbf{x}_i) , b(\mathbf{x}_i)) \mathcal{L}_{g_i} h_i^{\top}, $$ with  $a(\mathbf{x}_i) =  \mathcal{L}_{f_i}h_i(\mathbf{x}_i) + \alpha_i h_i(x_i), b(x_i)= \| \mathcal{L}_{g_i}h_i(x_i)\|^2$ and 
$ \lambda(a,b) =   \frac{-a + \sqrt{a^2 + 0.1b^2}}{2b}.$ Other options include to choose $k_i(\mathbf{x}_i)= 0$ or a nominal controller. Different choice of $k_i(\mathbf{x}_i)$ may differ in how difficult to satisfy the two conditions  \eqref{eq:CCBF_cond} and \eqref{eq:CCBF_cond_2}.
\end{remark}

As a simple illustrative example, consider a multi-agent system of 1-dimensional single-integrator agents, with the position of each agent given by $x_i \in \mathbb{R}$ and dynamics given by $\dot{x}_i = u_i$.
Let the coupling dynamic between neighboring agents be described using a formation control law
\begin{equation}\label{eq:u_f_i}
    u^f_i(\mathbf{x}_i) = -\xi \sum_{j \in \mathcal{N}_i} (x_i - x_j)  - \Delta_{ij}, 
\end{equation}
where $\xi>0$ and $\Delta_{ij} = x_i^\delta - x_j^\delta$ is the desired relative position between agents~$i$ and $j$, respectively, with $x_i^\delta \in \mathbb{R}$. In this sense, we can consider the formation control law as an induced drift term, which is a function of neighboring agent states, and the control input $u_i$ as any modification to this term as $\dot{x}_i =  u^f_i(\mathbf{x}_i) + u_i$, where, in our formulation in Section~\ref{sec:coupled_dyn}, this would make $f_i(\mathbf{x}_i) = u^f_i(\mathbf{x}_i)$ and $g_i = 1$.
For a safety constraint, consider a CCBF candidate $h_i$ where agent~$i$ is required to stay within a given distance $r > 0$ of the origin as $h_i(x_i) = \frac{1}{2}\left(r^2 - \Vert x_i \Vert^2\right).$
Consider the simplest example with two agents ($n=2$), where
\begin{equation} \label{eq:2agent_1D_dyn}
    \begin{aligned}
        \dot{x}_1 = u^f_1(x_1, x_2) + u_1, \;\; \dot{x}_2 = u^f_2(x_1, x_2),
    \end{aligned}
\end{equation}
and $h_2(x_2) = \frac{1}{2}\left(r^2 - x_2^2\right)$. In this example, since there is no input for agent~$2$, we set $k_2(x_1, x_2) = 0$, making
\begin{equation} \label{eq:h_2_plus_1D_ex}
    h_2^+(x_1, x_2) = \mathcal{L}_{f_2}h_2(x_1, x_2) + \alpha_2 h_2(x_2),
\end{equation}
where $\mathcal{L}_{f_2}h_2(x_1, x_2) = - x_2 u^f_2(x_1, x_2)$
and $\alpha_2$ is a positive scalar parameter for a linear class-$\mathcal{K}$ function. Thus, the derivative of $h_2^+$ is calculated as

\vspace{-3.5ex}
\small
\begin{align}
    \dot{h}_2^+(x_1, x_2, u_1) &= \mathcal{L}_{f_1} \mathcal{L}_{f_2}  h_2(x_1, x_2) + \mathcal{L}_{g_1} \mathcal{L}_{f_2}  h_2(x_1, x_2) u_1 \nonumber \\ 
    & \quad + \mathcal{L}_{f_2}^2 h_2(x_1, x_2) + \alpha_2 \mathcal{L}_{f_2}h_2(x_1, x_2), \label{eq:h_2_plus_dot_1D_ex}
\end{align}
\normalsize
Therefore, assuming that $\vert x_2(0) \vert < r$, in order for agent~$2$ to satisfy its safety constraint, agent~$1$ must choose $u_1$ according to \eqref{eq:h_2_plus_1D_ex} and \eqref{eq:h_2_plus_dot_1D_ex} such that
$\dot{h}_2^+(x_1, x_2, u_1) + \beta_2 h_2^+(x_1, x_2) \geq 0$,
where $\beta_2 > 0$. Note that since there is no $u_2$ and $k_2(x_1, x_2) = 0$, the second condition of the CCBF in \eqref{eq:CCBF_cond_2} reduces to a standard safety condition of $\gamma_2 h_2(x_2) \geq 0$ scaled by $\gamma_2>0$. 

In the scenario shown in Figure~\ref{fig:1D_example}, without any intervention by agent~$1$, agent~$2$ will violate its safety constraint as the dynamics \eqref{eq:u_f_i} will drive the position of agent~$2$ such that $x_2 > r$. However, under the CCBF condition, agent~$1$ will increase its velocity away from agent~$2$ such that $\Vert x_1-x_2 \Vert \geq \Delta$  before $x_2 \geq r$. We see this behavior demonstrated via simulation in Figure~\ref{fig:1D_example}.
In this example, computing a safe control input $u_1$ for the system is somewhat trivial; however, as the system complexity grows (i.e., if $d>1$ and $n>2$, with each agent having inputs $u_i$ and safety constraints $h_i$ for $i \in [n]$), determining both safe and optimal control inputs systematically for all agents in a distributed manner becomes a challenging and computationally complex problem. 

\begin{figure}
    \centering
    \begin{subfigure}[c]{.49\columnwidth}
        \begin{overpic} [width=\textwidth]{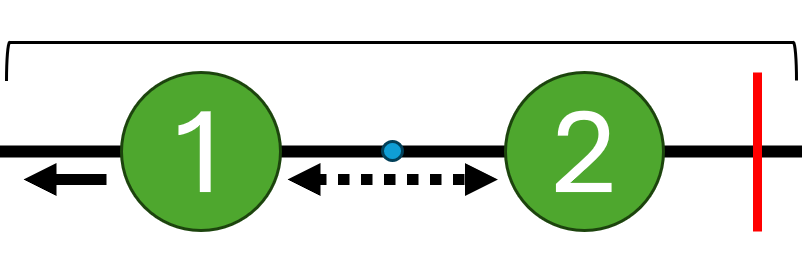}
            \put(46,29){$\Delta$}
            \put(92.5,-1){$r$}
            \put(7,-1){\small $u_1$}
            \put(32,-1){\small $u^f_i(x_1, x_2)$}
        \end{overpic}
    \end{subfigure}
    \begin{subfigure}[c]{.49\columnwidth}
        \includegraphics[trim= 13 10 10 10,clip, width=\textwidth]{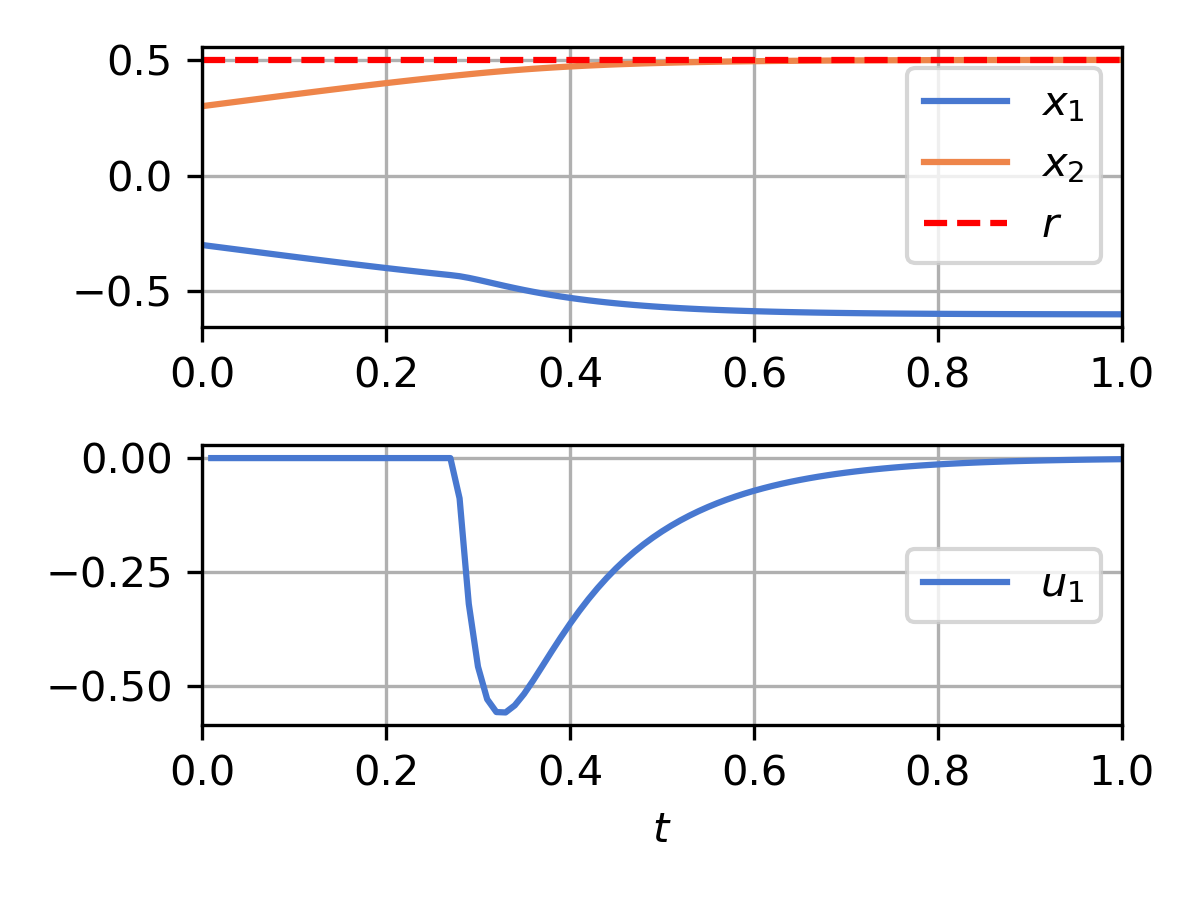}
    \end{subfigure}
    \captionsetup{belowskip=-15pt}
    \caption{(Left) An example of a simple 2-agent system where $d=1$ with dynamics defined by \eqref{eq:2agent_1D_dyn}. With no intervention by agent~$1$ through $u_1$, agent~$2$ will be driven by $u^f_2(x_1, x_2)$ to violate its safety constraint. (Right) A simulation of the example scenario shown on the left, with dynamics defined by \eqref{eq:2agent_1D_dyn} and agent~$1$ satisfying \eqref{eq:CCBF_cond} for all time, where $x_2(0) = -x_1(0) = 0.3$, $r = 0.5$, $\alpha_2 = \beta_2 = 10$, $\xi=2.5$, and $\Delta=1.4$.
    }
    \label{fig:1D_example}
\end{figure}  

\subsection{Distributed Optimization for Solving Safe Inputs}
In this subsection, we review a distributed optimization implementation from \cite{xiao2025continuous} to enforce the collaborative safety conditions. 
For notational simplicity, define
\begin{equation*}
\begin{aligned}
    & a_{ij}(\mathbf{x}_i) = \mathcal{L}_{g_j} \mathcal{L}_{f_i}  h_i(\mathbf{x}_i, x_j) +  \mathcal{L}_{g_j} (\mathcal{L}_{g_i}  h_i(x_i) k_i(\mathbf{x}_i)), \\
    & b_{ij}(\mathbf{x}_i,\mathbf{x}_j) = \mathcal{L}_{f_j} \mathcal{L}_{f_i}  h_i(\mathbf{x}_i, \mathbf{x}_j) +  \mathcal{L}_{f_j} (\mathcal{L}_{g_i}  h_i(x_i) k_i(\mathbf{x}_i)),
    \end{aligned}
\end{equation*}
for $i \neq j$, and
\begin{equation*}
    \begin{aligned}
    & a_{ii}(\mathbf{x}_i) = \mathcal{L}_{g_i} \mathcal{L}_{f_i}  h_i(\mathbf{x}_i) + \mathcal{L}_{g_i}^2  h_i(x_i) k_i(\mathbf{x}_i) + \alpha_i \mathcal{L}_{g_i}h_i(x_i) \\
        & b_{ii}(\mathbf{x}_i) = \mathcal{L}^2_{f_i} h_i(\mathbf{x}_i) + \mathcal{L}_{f_i} (\mathcal{L}_{g_i}  h_i(\mathbf{x}_i) k_i(\mathbf{x}_i)) \\
        &
        \quad \quad \quad \quad \quad \quad 
        + \alpha_i \mathcal{L}_{f_i}  h_i(\mathbf{x}_i) + \beta_i h_i^+(\mathbf{x}_i),
    \end{aligned}
\end{equation*}
for $i = j$. Here the term $ b_{ij}(\mathbf{x}_i,\mathbf{x}_j), j\in \mathcal{N}_i$ needs to be locally obtainable for agent~$i$. Assume that the agent~$i$ can communicate with its $2$-hop neighbors, where, recall, the notion of neighbors is defined based on coupling dynamics.

Then, the left hand side of \eqref{eq:CCBF_cond} can be rewritten as
\begin{equation}\label{eq:collab_safety_cond}
    \psi_i(\mathbf{x}_i^+,u_i,u_{\mathcal{N}_i^+}) = \sum_{j \in \mathcal{N}_i^+} a_{ij}(\mathbf{x}_i)^\top u_j  +  b_{ij}(\mathbf{x}_i,\mathbf{x}_j).
\end{equation}

 One common optimization-based safety filter design is 
\begin{equation} \label{eq:prob_statement}
    \begin{aligned}
        \min_{\mathbf{u}} & \quad \sum_{i \in [n]}  \Vert u_i - u_{i}^{nom} \Vert^2 \\
         \text{s.t} & \quad \psi_i(\mathbf{x}_i^+,u_i,u_{\mathcal{N}_i^+}) \geq 0 \\
         & \quad
         \mathcal{L}_{g_i}h_i(x_i) u_i +  \gamma_i h_i(x_i) \geq \mathcal{L}_{g_i}h_i(x_i) k_i(\mathbf{x}_i) , 
         \forall i \in [n]
    \end{aligned}
\end{equation}
where $\mathbf{u} = [u_{1}^\top; \cdots ;u_{n}^\top]$ and $u_i^{nom}$ represents the nominal control input for agent $i$. This formulation fits in the problem setting of \cite{xiao2025continuous}. Based on \cite[Proposition~1]{xiao2025continuous}, by introducing a set of auxiliary decision variables $y_j^i \in \mathbb{R}$ relating node~$j$ to the safety of node~$i$, we can describe \eqref{eq:prob_statement} with an equivalent quadratic problem as 
\begin{equation} \label{eq:prob_statement_aux_var}
    \begin{aligned}
        \min_{\mathbf{u}, \mathbf{y}_1, \dots, \mathbf{y}_n} & \quad \Vert \mathbf{u} - \mathbf{u}^{nom} \Vert^2 \\
         \text{s.t} & \quad \mathbf{a}_1^\top \mathbf{u}  + \sum_{j \in \mathcal{N}_1^+} (y_1^1 - y_j^1) + \mathbb{1}^\top \mathbf{b}_{1} \geq 0 \\
         & \quad \quad \quad \vdots \\
         & \quad 
         \mathbf{a}_n^\top \mathbf{u}  + \sum_{j \in \mathcal{N}_n^+} (y_n^n - y_j^n) +  \mathbb{1}^\top \mathbf{b}_{n} \geq 0 \\
         & \quad G \mathbf{u} \geq \mathbf{q},
    \end{aligned}
\end{equation}
where $\mathbf{a}_i^\top = [a_{i1}(\mathbf{x}_1)^\top; \cdots ;a_{in}(\mathbf{x}_n)^\top]$, $\mathbf{b}_{i} = [b_{i1}(\mathbf{x}_i,\mathbf{x}_1); \cdots ; b_{in}(\mathbf{x}_i,\mathbf{x}_n)]^\top$, with $a_{ij} = \mathbf{0}$ and $b_{ij} = 0$ if $j \notin \mathcal{N}_i^+$, $G = \text{blkdiag}([\mathcal{L}_{g_i}h_i(x_i)]_{i \in [n]})$, $\mathbf{q} = [\mathcal{L}_{g_i}h_i(x_i) k_i(\mathbf{x}_i)-\gamma_i h_i(x_i)]_{i \in [n]}$.

Moreover, using techniques developed in \cite{xiao2025continuous}, we can solve the above optimization problem in a distributed manner: for   each agent~$i \in [n]$, it solves the following problem

\vspace{-2ex}
\small
\begin{align}
    &\phi_i(\mathbf{y}) := \min_{u_i \mathbf{y}_i} \quad \Vert u_i - u_{i}^{nom} \Vert^2 \nonumber\\
     \text{s.t.} & \quad a_{ji}(\mathbf{x}_j)^\top u_i  + \sum_{k \in \mathcal{N}_j^+} (y_i^j - y_k^j) +  b_{ji}(\mathbf{x}_j,\mathbf{x}_i) \geq 0
     ;\forall j \in \mathcal{N}_i^-, \nonumber\\
     & \quad
     \mathcal{L}_{g_i}h_i(x_i) u_i +  \gamma_i h_i(x_i) \geq \mathcal{L}_{g_i}h_i(x_i) k_i(\mathbf{x}_i) \label{eq:prob_statement_distr}
\end{align}
\normalsize
and update the variables $y_i^j$ using the following update law
\begin{equation} \label{eq:y_updates}
    \dot{y}_i^j = 
    \begin{cases}
        -k_0 \sum_{k \in \mathcal{N}_j^+} (c_i^j - c_k^j) & \text{if } j \in \mathcal{N}_i^- \\
        0 & \text{if } j \notin \mathcal{N}_i^-
    \end{cases}
\end{equation}
Here $k_0>0$ is a gain parameter, $c_i = (c_i^1, \dots, c_i^n)$ where $c_i^j$ is the Lagrange multiplier of \eqref{eq:prob_statement_distr} that corresponds to the condition $ a_{ji}(\mathbf{x}_j)^\top u_i  + \sum_{k \in \mathcal{N}_i \cap \mathcal{N}_j^+} (y_i^j - y_k^j) +  b_{ji}(\mathbf{x}_j,\mathbf{x}_i) \geq 0$ if $j \in \mathcal{N}_i^-$, and is equal to $0$ otherwise. We note that \eqref{eq:prob_statement_distr} and \eqref{eq:y_updates} are distributed in nature since each node $i$ can communicate  $y^j_k$s and $c^j_k$s with its $2$-hop neighbors $k$.

\begin{lemma}[\hspace{-0.1pt}\cite{xiao2025continuous}] \label{lem:distributed_optimization}
    Suppose that the local optimization problem \eqref{eq:prob_statement_distr} is always feasible and that the solution to the ODE \eqref{eq:y_updates} is forward complete. Then
    \begin{enumerate}
        \item the locally computed inputs $u_i$ from \eqref{eq:prob_statement_distr} satisfy the constraints in the centralized problem \eqref{eq:prob_statement} for all time;
        \item When viewing  \eqref{eq:prob_statement} as a static optimization problem, the locally computed inputs $u_i(t)\to u_i^\star$ as $t\to\infty$ where $u_i^\star$ denotes the optimal input to the central problem. Moreover, the convergence rate can be adjusted by choosing a different $k_0$.
    \end{enumerate}
\end{lemma}
The above results are from \cite[Theorem~2]{xiao2025continuous}. One practical approach to apply this distributed result is to solve \eqref{eq:prob_statement_distr} and \eqref{eq:y_updates} at a faster rate, and execute $u_i$ at a slower rate.

\subsection{\textbf{Example Continued:} Two Safety Constraints} \label{sec:sub:two_constraints}
Returning to our example from Section~\ref{sec:1D_example}, let us now consider the case where
\begin{equation} \label{eq:2agent_1D_dyn_2inputs}
    \begin{aligned}
        \dot{x}_1 = u^f_1(x_1, x_2) + u_1, \;\; \dot{x}_2 = u^f_2(x_1, x_2) + u_2,
    \end{aligned}
\end{equation}
with both $h_1(x_1) = \frac{1}{2}\left(r^2 - x_1^2\right)$ and $h_2(x_2) = \frac{1}{2}\left(r^2 - x_2^2\right)$. In this case, we can utilize the distributed optimization framework from \eqref{eq:prob_statement_distr} to solve for the optimal safe action that satisfies both agent safety constraints.  
We test two choices of $k_i(\mathbf{x}_i)$ in simulation; namely, $k_i = 0$ and $k_i(\mathbf{x}_i)=  \lambda(a(\mathbf{x}_i) , b(\mathbf{x}_i)) \mathcal{L}_{g_i} h_i^{\top}$ (i.e., the ``half-Sontag" control law described in Remark~\ref{rem:k_i}), which are shown in in Figure~\ref{fig:1D_example_ki}. Note that in Figure~\ref{fig:sub:ki_0}, the choice of $k_i = 0$ induces conservative behavior for both agents, i.e., there is a gap between the position of the agents and the safety boundary. In Figure~\ref{fig:sub:ki_hsontag} when we choose $k_i(\mathbf{x}_i)=  \lambda(a(\mathbf{x}_i) , b(\mathbf{x}_i)) \mathcal{L}_{g_i} h_i^{\top}$ (which could be considered the solution to the smallest $u_i$ that keeps agent~$i$ safe in the first derivative of $h_i$), each agent approaches the safety boundary and the distributed optimization problem always remains feasible through the local auxiliary variable updates.

 However, the above approach requires the feasibility of the local optimization problem~\eqref{eq:prob_statement_distr}, which can be restrictive in many scenarios. In Section~\ref{sec:altruistic_safety}, we propose a method for handling this restriction from an altruistic safety perspective.

\begin{figure}
    \centering
    \begin{subfigure}[t]{.49\columnwidth}
        \includegraphics[trim= 13 10 10 10,clip, width=\textwidth]{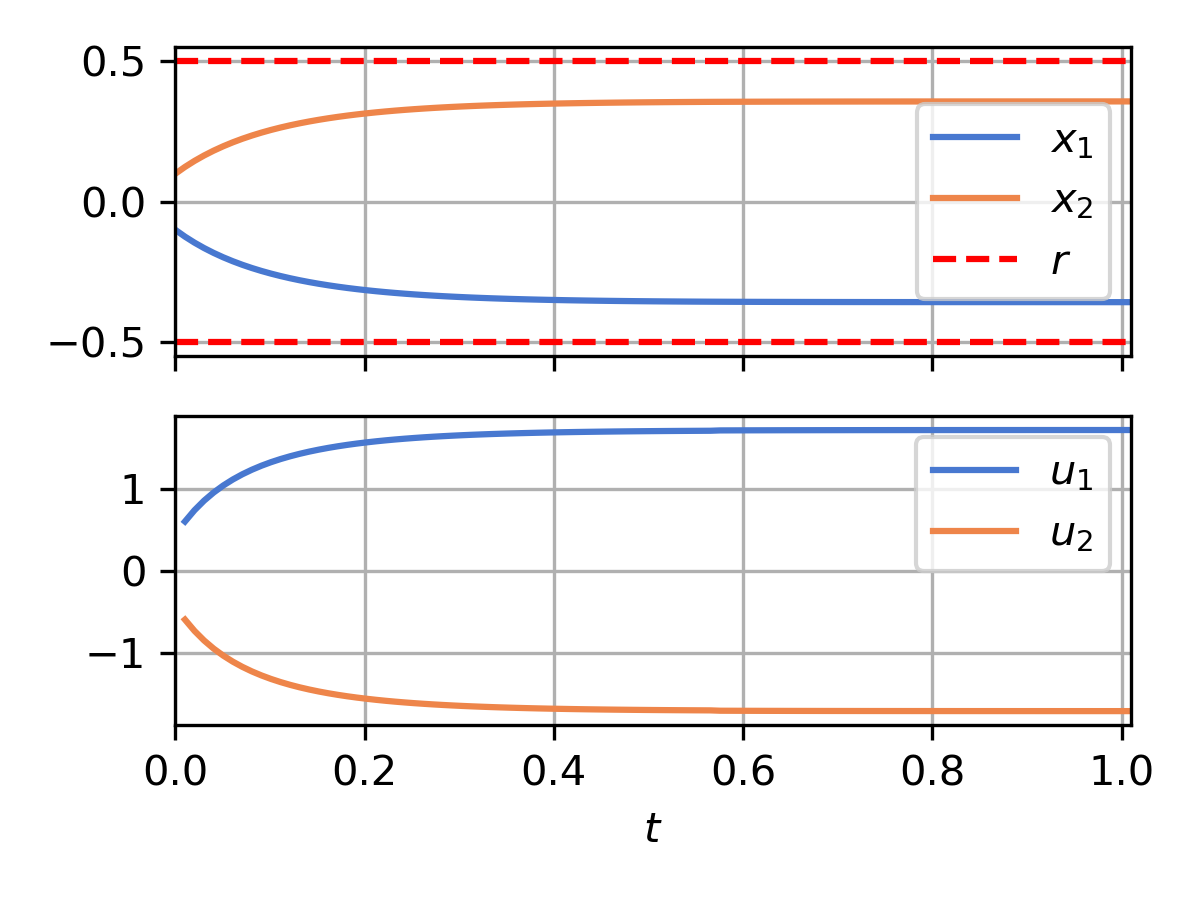}
        \caption{}
        \label{fig:sub:ki_0}
    \end{subfigure}
    \begin{subfigure}[t]{.49\columnwidth}
        \includegraphics[trim= 13 10 10 10,clip, width=\textwidth]{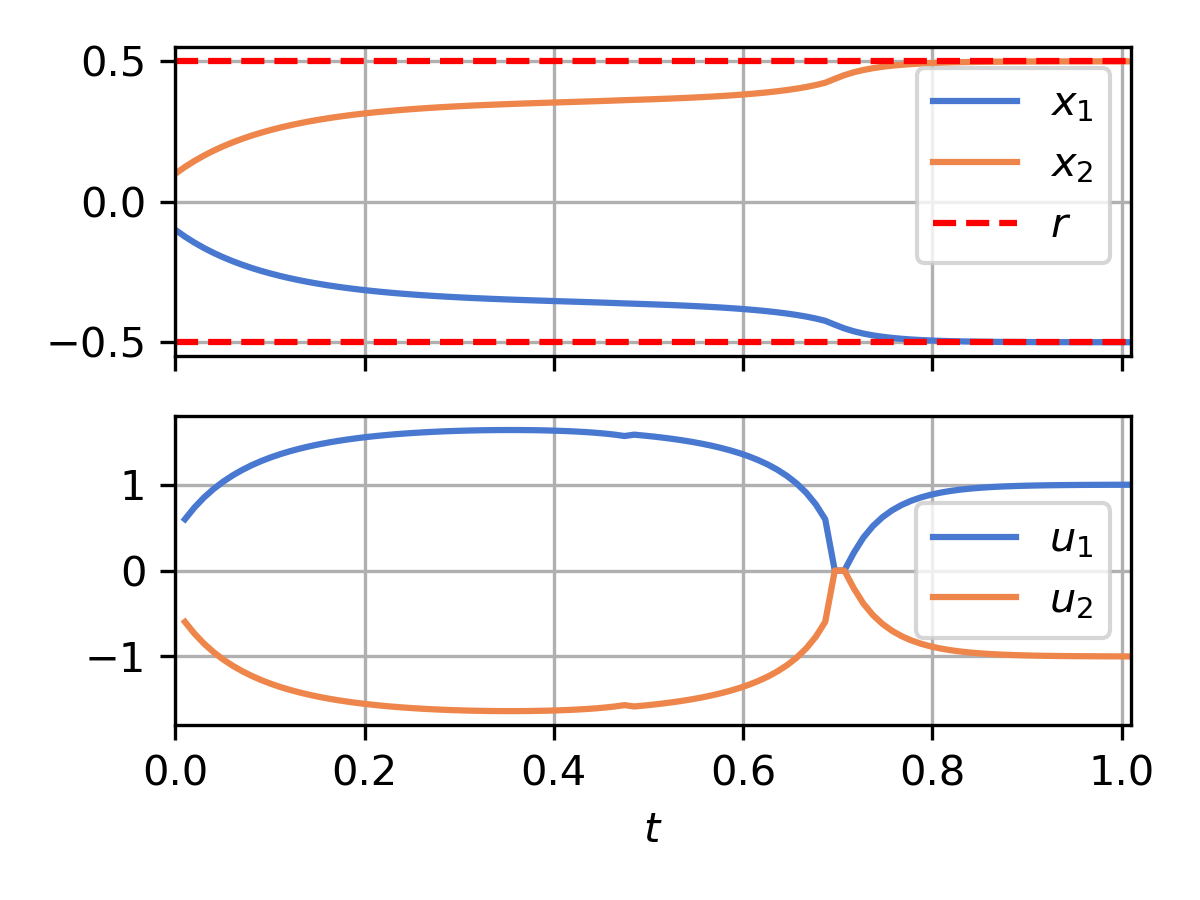}
        \caption{}
        \label{fig:sub:ki_hsontag}
    \end{subfigure}
    \captionsetup{belowskip=-15pt}
    \caption{A simulation with the same initial conditions and parameterization as the simulation in Figure~\ref{fig:1D_example}, but with dynamics defined by \eqref{eq:2agent_1D_dyn_2inputs}. In (a), $k_1(x_1,x_2) = k_2(x_1,x_2) = 0$, whereas in (b), $k_i(\mathbf{x}_i)$ is defined by the ``half-Sontag" feedback control law for both agents, as described in Remark~\ref{rem:k_i}. The control inputs $u_1, u_2$ are obtained by solving \eqref{eq:prob_statement_distr} with auxiliary variable updates defined by \eqref{eq:y_updates}.
    }
    \label{fig:1D_example_ki}
\end{figure}

\section{Altruistic Safety in Coupled Systems} \label{sec:altruistic_safety}
In ecology, \textit{genetic fitness} is used to quantify the reproductive success of a given organism \cite{orr2009fitness}. \textit{Inclusive fitness} extends this notion of genetic success to include organisms with the same, or similar, genes, where individuals may take \textit{altruistic} actions that support the survival of each other, even at a cost to themselves, to theoretically enhance the genetic fitness of both the recipient of the act and the altruistic organism. Hamilton's rule \cite{hamilton1963evolution} underpins this theory of inclusive fitness by deriving a condition under which altruistic genes are likely to propagate throughout a given population as $r_{ij} B_j(u_i) \geq C_i(u_i)$,
where, typically, $C_i(u_i)$ is the reproductive cost to organism $i$ of a given choice $u_i$, $B_j(u_i)$ is the benefit of the choice $u_i$ to organism $j$, and $r_{ij}$ is the genetic relatedness between the to organisms, which, in the ecological setting, {means} the probability that $i$ and $j$ share the same genes.

In previous work \cite{butler2025hamilton}, rather than using fitness to describe agent \textit{fecundity} (i.e., reproductive rate) of, agent fitness is instead considered as a measure of \textit{productivity} (i.e., task/goal completion rate \cite{nguyen2024scalable}). However, in this work, we seek to extend this same notion of fitness to agent safety, where agents may choose to take on additional safety risks to support the safety of more critical neighbors.

\subsection{Safety-Relatedness in Coupled Systems}
We consider the relatedness between nodes in the context of quantifying safety importance through
agent safety sensitivity. 
Let each agent receive a weight $w_i \geq 0$ that denotes importance based on the sensitivity of a given agent to individual failure. One way we could compute this is to evaluate how close a given agent is to violating its defined safety condition $h_i$ as
\begin{equation} \label{eq:w_i^s}
    w_i(x_i) = \frac{\eta_i}{h_i(x_i)},
\end{equation}
where $\eta_i \geq 0$ may be considered a scalar bias towards the safety of agent~$i$. 
Note that this metric of safety importance is not exhaustive and could be further designed to reflect the specific safety needs of a given system.
We can compute the relatedness between agents in the context of safety importance as
\begin{equation} \label{eq:r_ij}
    r_{ij} = \frac{w_j(x_j)}{w_i(x_i)},
\end{equation}
where $r_{ij} = 1$ implies the safety of agent~$i$ and $j$ is equally important, $0 < r_{ij} < 1$ implies agent~$i$'s safety is more critical relative to agent~$j$, and vice versa for $r_{ij} > 1$.

\subsection{Conditions for Altruistic Safety}
Using the collaborative safety condition from \eqref{eq:collab_safety_cond}, derived in Section~\ref{sec:coupled_safety_cond}, we can construct Hamilton's rule-like candidate functions that describe the safety cost and benefit of a given input $u_i$ concerning the safety for agent~$i$ and its neighbors $j\in \mathcal{N}_i^-$, respectively, as follows
\begin{equation}\label{eq:safety_cost_i}
    C_i(\mathbf{x}_i,u_i) := -a_{ii}(\mathbf{x}_i)^\top u_i - b_{ii}(\mathbf{x}_i)
\end{equation}
\begin{equation}\label{eq:safety_benefit_ij}
    B_{ij}(\mathbf{x}_i,\mathbf{x}_j,u_i) := a_{ji}(\mathbf{x}_j)^\top u_i + b_{ji}(\mathbf{x}_j,\mathbf{x}_i), 
\end{equation}
where $C_i(\mathbf{x}_i,u_i) \geq 0$ implies $u_i$ contributes positively (i.e., a negative cost) towards satisfying its safety condition \eqref{eq:collab_safety_cond} and $B_{ij}(\mathbf{x}_i,\mathbf{x}_j,u_i) \geq 0$ implies $u_i$ contributes positively to agent~$j$'s safety condition.
Thus, using \eqref{eq:safety_cost_i} and \eqref{eq:safety_benefit_ij}, we construct an altruistic safety condition for agent~$i$ with respect to neighbors $j \in \mathcal{N}_i^-$ as
\begin{equation}\label{eq:hams_rule_N_i-}
    \sum_{j \in \mathcal{N}_i^-} r_{ij}B_{ij}(\mathbf{x}_i,\mathbf{x}_j,u_i) \geq C_i(\mathbf{x}_i,u_i).
\end{equation}
Due to the linear coupling constraints defined in Section~\ref{sec:coupled_safety_cond}, we can express \eqref{eq:hams_rule_N_i-} as a linear inequality with respect to $u_i$,
\begin{equation} \label{eq:hams_rule_lin_eq}
    \begin{aligned}
       \left( \sum_{j \in \mathcal{N}_i^-} r_{ij}  a_{ji}(\mathbf{x}_j)  \right)^\top u_i 
        \geq  - \sum_{j \in \mathcal{N}_i^-} r_{ij} b_{ji}(\mathbf{x}_j,\mathbf{x}_i),
    \end{aligned}   
\end{equation}
where, again, note that $i \in \mathcal{N}_i^-$ and $r_{ii} = \frac{w_i}{w_i} = 1$.

Consider the set of safe inputs for agent~$i$ according to the collaborative safety condition from \eqref{eq:collab_safety_cond} given a set of neighbor inputs $u_{\mathcal{N}_i}$
\begin{equation} 
    \mathcal{U}_i^s(\mathbf{x}_i^+, u_{\mathcal{N}_i}) = \{ u_i \in \mathbb{R}^{M_i}: \psi_i(\mathbf{x}_i^+, u_i, u_{\mathcal{N}_i}) \geq 0 \}.
\end{equation}
Note that by \eqref{eq:collab_safety_cond}, any $u_i \in \mathcal{U}_i^s(\mathbf{x}_i^+, u_{\mathcal{N}_i})$  must satisfy

\small
\begin{equation}\label{eq:safe_cond_separated}
    \begin{aligned}
        a_{ii}(\mathbf{x}_i)^\top u_i + b_{ii}(\mathbf{x}_i) + \sum_{j \in \mathcal{N}_i^+ \setminus \{i\}} a_{ij}(\mathbf{x}_i)^\top u_j  +  b_{ij}(\mathbf{x}_i,\mathbf{x}_j) \geq 0,
    \end{aligned}
\end{equation}
\normalsize
given a set of neighbor inputs $u_{\mathcal{N}_i}$.
Further, consider the set of inputs for agent~$i$ that satisfy the altruistic safety condition for neighbors $j \in \mathcal{N}_i^-$
\begin{equation}
    \mathcal{U}_i^a(\mathbf{x}_i^+) = \{ u_i \in \mathbb{R}^{M_i}: \eqref{eq:hams_rule_lin_eq} \}.
\end{equation}

We first show that, so long as an agent's safety importance is large when compared with its neighbors, any input $u_i$ that satisfies \eqref{eq:hams_rule_lin_eq} 
also yield a safe action for agent~$i$.

\begin{proposition} \label{prop:alt_is_safe}
When $w_i(x_i) \gg w_j(x_j), \forall j \in \mathcal{N}_j^-$: 
\begin{enumerate}
    \item For any $u_{\mathcal{N}_i}$ such that
    $$
    \sum_{j \in \mathcal{N}_i^+ \setminus \{i\}} a_{ij}(\mathbf{x}_i)^\top u_j  +  b_{ij}(\mathbf{x}_i,\mathbf{x}_j) \geq 0,
    $$
    we have $\mathcal{U}_i^a(\mathbf{x}_i^+) \subseteq \mathcal{U}_i^s(\mathbf{x}_i^+, u_{\mathcal{N}_i})$.
    \item For any $u_{\mathcal{N}_i}$ such that $\mathcal{U}_i^s(\mathbf{x}_i^+, u_{\mathcal{N}_i}) \neq \emptyset$, we have $\mathcal{U}_i^a(\mathbf{x}_i^+) \cap \mathcal{U}_i^s(\mathbf{x}_i^+, u_{\mathcal{N}_i}) \neq \emptyset$.
\end{enumerate}
\end{proposition}
\begin{proof}
    When $w_i(x_i) \gg w_j(x_j), \forall j \in \mathcal{N}_j^-$ we have by \eqref{eq:r_ij} that the set $\mathcal{U}_i^a(\mathbf{x}_i^+)$  reduces to
    \begin{equation}\label{eq:lem_alt_cond}
      \mathcal{U}_i^a(\mathbf{x}_i^+) = \{ u_i \in \mathbb{R}^{M_i}:  a_{ii}(\mathbf{x}_i)^\top u_i + b_{ii}(\mathbf{x}_i) + \epsilon \geq 0\},
    \end{equation}
    where $\epsilon$ is arbitrarily small in magnitude. Therefore, when $\sum_{j \in \mathcal{N}_i^+ \setminus \{i\}} a_{ij}(\mathbf{x}_i)^\top u_j  +  b_{ij}(\mathbf{x}_i,\mathbf{x}_j) \geq 0$,
    for any $u_i \in  \mathcal{U}_i^a(\mathbf{x}_i^+)$, it follows that $a_{ii}(\mathbf{x}_i) u_i + b_{ii}(\mathbf{x}_i,\mathbf{x}_j) + \sum_{j \in \mathcal{N}_i^+ \setminus \{i\}} a_{ij}(\mathbf{x}_i)^\top u_j  +  b_{ij}(\mathbf{x}_i,\mathbf{x}_j) \geq 0$, i.e., $u_i \in \mathcal{U}_i^s(\mathbf{x}_i^+,u_{\mathcal{N}_i})$. This proves
    Point~1.
    Moreover, 
    in this case, 
    since $\mathcal{U}_i^a(\mathbf{x}_i^+)$ is by definition a closed and non-empty halfspace in $\mathbb{R}^{M_i}$, we know $\mathcal{U}_i^a(\mathbf{x}_i^+) \cap \mathcal{U}_i^s(\mathbf{x}_i^+, u_{\mathcal{N}_i}) \neq \emptyset$. When $\sum_{j \in \mathcal{N}_i^+ \setminus \{i\}} a_{ij}(\mathbf{x}_i)^\top u_j  +  b_{ij}(\mathbf{x}_i,\mathbf{x}_j) < 0$, in view of \eqref{eq:safe_cond_separated} and \eqref{eq:lem_alt_cond}, it follows that $\mathcal{U}_i^s(\mathbf{x}_i^+, u_{\mathcal{N}_i}) \subset \mathcal{U}_i^a(\mathbf{x}_i^+)$. Thus Point~2 holds.   
\end{proof}

When an agent $i$ approaches the boundary of its safety constraint $\mathcal{C}_i$, $w_i(x_i)$ grows arbitrarily large by \eqref{eq:w_i^s}. However, we note that $w_i(x_i)$ in~\eqref{eq:w_i^s}, and consequently \eqref{eq:r_ij} and \eqref{eq:hams_rule_lin_eq}, are undefined at $h_i(x_i) = 0$.
We now show that if an agent~$i$ has a significantly greater safety importance when compared with the safety importance of other agents in its 2-hop neighborhood, the altruistic safety condition \eqref{eq:hams_rule_lin_eq} will promote a larger set of feasible safe inputs for agent~$i$ through a new altruistic distributed optimization problem:

\vspace{-2ex}
\small
\begin{align}
    &\phi_i^a(\mathbf{y}) := \min_{u_i \mathbf{y}_i} \quad \Vert u_i - u_{i}^* \Vert^2 \nonumber \\
    & \text{s.t.}  \quad a_{ji}(\mathbf{x}_j)^\top u_i  + \sum_{k \in \mathcal{N}_j^+} (y_i^j - y_k^j) +  b_{ji}(\mathbf{x}_j,\mathbf{x}_i) \geq 0
     ;\forall j \in \mathcal{N}_i^-, \nonumber\\
     & \quad  \quad
     \mathcal{L}_{g_i}h_i(x_i) u_i +  \gamma_i h_i(x_i) \geq \mathcal{L}_{g_i}h_i(x_i) k_i(\mathbf{x}_i), \nonumber \\
     & \quad  \quad
    \sum_{j \in \mathcal{N}_i^-} r_{ij}  a_{ji}(\mathbf{x}_j)^\top u_i 
    \geq  - \sum_{j \in \mathcal{N}_i^-} r_{ij} b_{ji}(\mathbf{x}_j,\mathbf{x}_i). \label{eq:prob_statement_distr_alt}
\end{align}
\normalsize

\begin{theorem} \label{thm:alt_feasibility}
 For a given system state $x_i, \forall i \in [n]$,  let  $\mathbf{u}^{o*}, \mathbf{y}^{o*}$ and $\mathbf{u}^{a*}, \mathbf{y}^{a*}$ be the respective optimal solutions for  \eqref{eq:prob_statement_distr} and \eqref{eq:prob_statement_distr_alt}, both using the update law \eqref{eq:y_updates}.  If $w_i(\mathbf{x}_i^+) \gg w_k(\mathbf{x}_k^+), \forall k \in \bigcup_{j \in \mathcal{N}_i^+} \mathcal{N}_j^-$, then
    \begin{equation} \label{eq:thm1_set_inclusion}
        \mathcal{U}_i^s(\mathbf{x}_i^+, u_{\mathcal{N}_i}^{o*}) \subseteq \mathcal{U}_i^s(\mathbf{x}_i^+, u_{\mathcal{N}_i}^{a*}),
    \end{equation}
    where $u_{\mathcal{N}_i}^{o*}$ and $u_{\mathcal{N}_i}^{a*}$ are the optimal solutions for neighbors $j \in \mathcal{N}_i^+$ according to \eqref{eq:prob_statement_distr} and \eqref{eq:prob_statement_distr_alt}, respectively. 
\end{theorem}
\begin{proof}
     From \eqref{eq:hams_rule_lin_eq}, the additional condition in the altruistic distributed optimization for agent~$j \in \mathcal{N}_i^+$ is
    \begin{equation}\label{eq:thm1_altruism_condition}
       \begin{aligned}
            &r_{ji} (a_{ij}(\mathbf{x}_i)^\top u_j^a + b_{ij}(\mathbf{x}_i,\mathbf{x}_j)) \\ 
            &+ \sum_{k \in \mathcal{N}_j^- \setminus \{i\}} r_{jk} \left(a_{kj}(\mathbf{x}_j)^\top u_j^a + b_{kj}(\mathbf{x}_k,\mathbf{x}_j)\right) \geq 0, 
       \end{aligned}
    \end{equation}
    Multiplying \eqref{eq:thm1_altruism_condition} by $\frac{1}{r_{ji}}$ yields
    \begin{equation*}
      \begin{aligned}
           & a_{ij}(\mathbf{x}_i)^\top u_j^a + b_{ij}(\mathbf{x}_i,\mathbf{x}_j) \\ 
            & + \sum_{k \in \mathcal{N}_j^-\setminus \{i\}} \frac{w_k(\mathbf{x}_k^+)}{w_i(\mathbf{x}_i^+)} \left(a_{kj}(\mathbf{x}_j)^\top u_j^a + b_{kj}(\mathbf{x}_k,\mathbf{x}_j)\right) \geq 0. 
       \end{aligned}
    \end{equation*}
    Thus, if $w_i(\mathbf{x}_i^+) \gg w_k(\mathbf{x}_k^+), \forall k \in \bigcup_{j \in \mathcal{N}_i^+} \mathcal{N}_j^-$, then any feasible solution $u_j^a$ must satisfy
    \begin{equation} \label{eq:thm1_alt_contribution}
        a_{ij}(\mathbf{x}_i)^\top u_j^a + b_{ij}(\mathbf{x}_i,\mathbf{x}_j) + \epsilon \geq 0, \forall j \in  \mathcal{N}_i^+,
    \end{equation}
    where $\epsilon$ is arbitrarily small in magnitude. 
    
    Now consider the respective optimal solutions $u_{\mathcal{N}_i}^{o*}$ and $u_{\mathcal{N}_i}^{a*}$. Recall that both solutions share the same cost functions and constraints except for the additional one in \eqref{eq:thm1_alt_contribution}.  In order to show \eqref{eq:thm1_set_inclusion}, based on the set definition in \eqref{eq:safe_cond_separated} we need to demonstrate that, compared to the optimal formulation in \eqref{eq:prob_statement_distr}, the neighboring agents in the altruistic distributed optimization contribute more to the safety condition of agent~$i$. Mathematically, we need to show
    \begin{equation} \label{eq:thm1_key_inequality}
        \sum_{j \in \mathcal{N}_i^+ \setminus \{i\}} a_{ij}(\mathbf{x}_i)^\top u_j^{o*}  \leq \sum_{j \in \mathcal{N}_i^+ \setminus \{i\}} a_{ij}(\mathbf{x}_i)^\top u_j^{a*}.
    \end{equation}

    Suppose that $a_{ij}(\mathbf{x}_i)^\top u_j^{o*} + b_{ij}(\mathbf{x}_i,\mathbf{x}_j)  \geq 0, \forall j \in  \mathcal{N}_i^+, $ then the additional constraints in \eqref{eq:thm1_alt_contribution} are trivially satisfied by $u_j^{o*}, \forall j \in  \mathcal{N}_i^+$. It follows that $ u_j^{o*} =  u_j^{a*}$ and the inequality in \eqref{eq:thm1_key_inequality} holds. 

    Now suppose that $a_{ij}(\mathbf{x}_i)^\top u_j^{o*} + b_{ij}(\mathbf{x}_i,\mathbf{x}_j)  <  0$ for some  agents. Let $\mathcal{M}_{-}\subseteq  \mathcal{N}_i^+$ be the index set and $\mathcal{M}_{+} = \mathcal{N}_i^+ \setminus \mathcal{M}_{-} $. We consider an optimization problem with the additional constraints \eqref{eq:thm1_alt_contribution} for agents in  $\mathcal{M}_{-}$, and denote the optimal solution to this problem as $\mathbf{u}^{m*}$. If we fix $\mathbf{u}_k$ to be $ \mathbf{u}_k= \mathbf{u}_k^{m*}, k \in \mathcal{M}_{-}$, as no additional conditions are posed on  $\mathbf{u}_{j}, j\in \mathcal{M}_{+}$ and the coupling constraint in \eqref{eq:safe_cond_separated} gets easier to satisfy, the feasible set for $\mathbf{u}_{j}, j\in \mathcal{M}_{+}$ gets larger. Depending on the positiveness of $a_{ij}(\mathbf{x}_i)^\top u_j^{m*} + b_{ij}(\mathbf{x}_i,\mathbf{x}_j) $ for $j\in \mathcal{M}_{+}$, we can repeat the above process and show that the feasible set for $\mathbf{u}_{j}, j\in \mathcal{M}_{2,+}$ gets even larger. When this process terminates, the optimal solution $\mathbf{u}^{m*} = \mathbf{u}^{a*}$, and the feasible set for $\mathbf{u}_{i}$ have this set inclusion property. 
\end{proof}

\subsection{\textbf{Example Continued:} Feasibility with Altruism}
Returning again to our example from the previous section, we illustrate how adding the altruism condition in the distributed optimization problem \eqref{eq:prob_statement_distr_alt} promotes feasibility for agents with higher importance. In Figure~\ref{fig:sub:alt_example}, we show a simulation with the same parameters and initial conditions from Section~\ref{sec:sub:two_constraints}, but with each agent solving for the optimal safe action according to \eqref{eq:prob_statement_distr_alt}, where $\eta_1 = 1$ and $\eta_2 = 1000$, where we note a difference in behavior during $t \in  [0.4, 0.6]$. In this case with one control input, we can quantify the feasibility for agent~2 by computing $u_{2}^{min} = \frac{-(b_{22} + a_{21}u_1^*+b_{21})}{a_{22}}$, where $u_{2}^{min}$ is the minimum $u_2$ that satisfies $\mathcal{U}_2^s(x_1, x_2, u_1)$. In Figure~\ref{fig:sub:u_min_diff}, we compare $u_{2}^*$ between two simulations where $u_1, u_2$ are computed using \eqref{eq:prob_statement_distr} and \eqref{eq:prob_statement_distr_alt} and plot the difference between $u_{2}^{min}$ in both cases, where we see that \eqref{eq:prob_statement_distr_alt} with $\eta_1 = 1$ and $\eta_2 = 1000$ does indeed yield a larger set of feasibly safe actions for agent~2 while still maintaining safety for agent~1.

\begin{figure}
    \centering
    \begin{subfigure}[t]{.49\columnwidth}
        \includegraphics[trim= 13 10 10 10,clip, width=\textwidth]{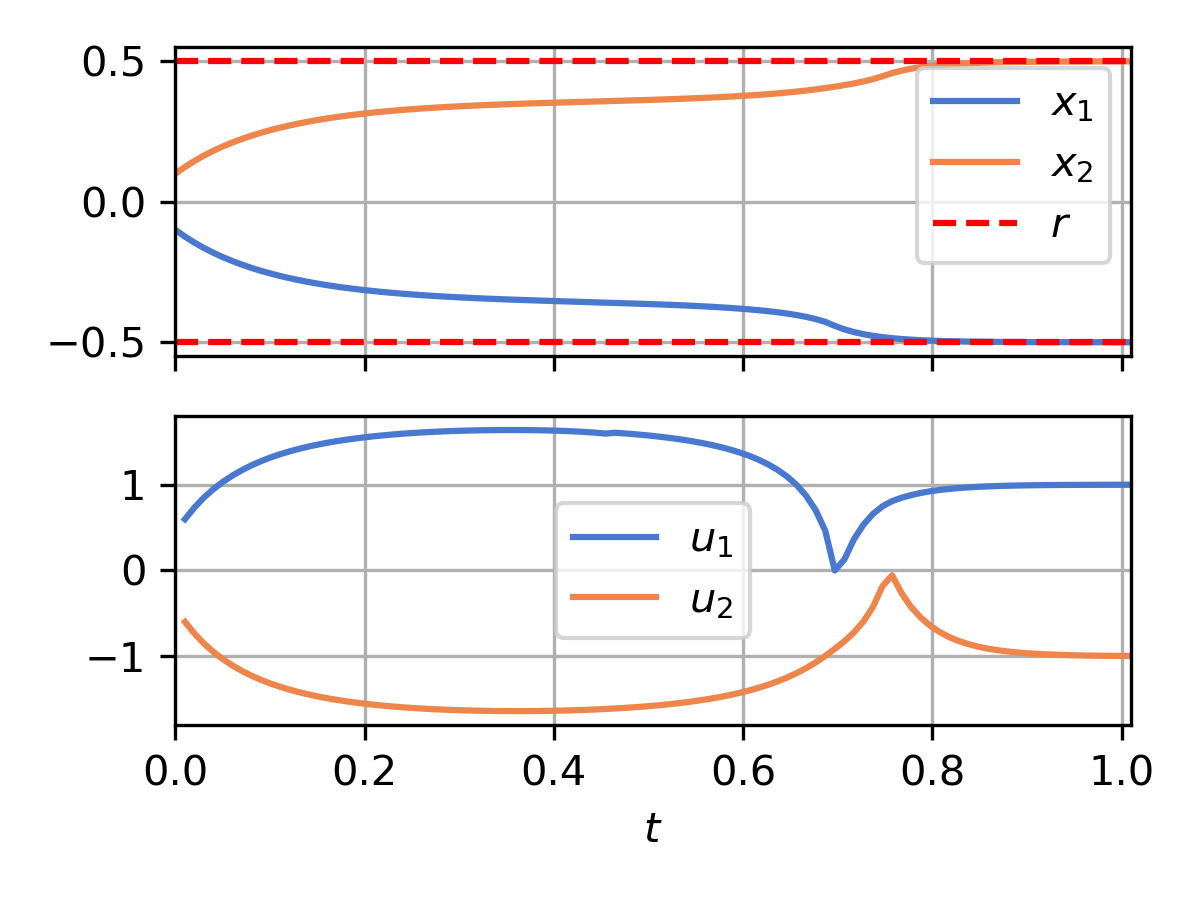}
        \caption{}
        \label{fig:sub:alt_example}
    \end{subfigure}
    \begin{subfigure}[t]{.49\columnwidth}
        \includegraphics[trim= 12 10 10 10,clip, width=\textwidth]{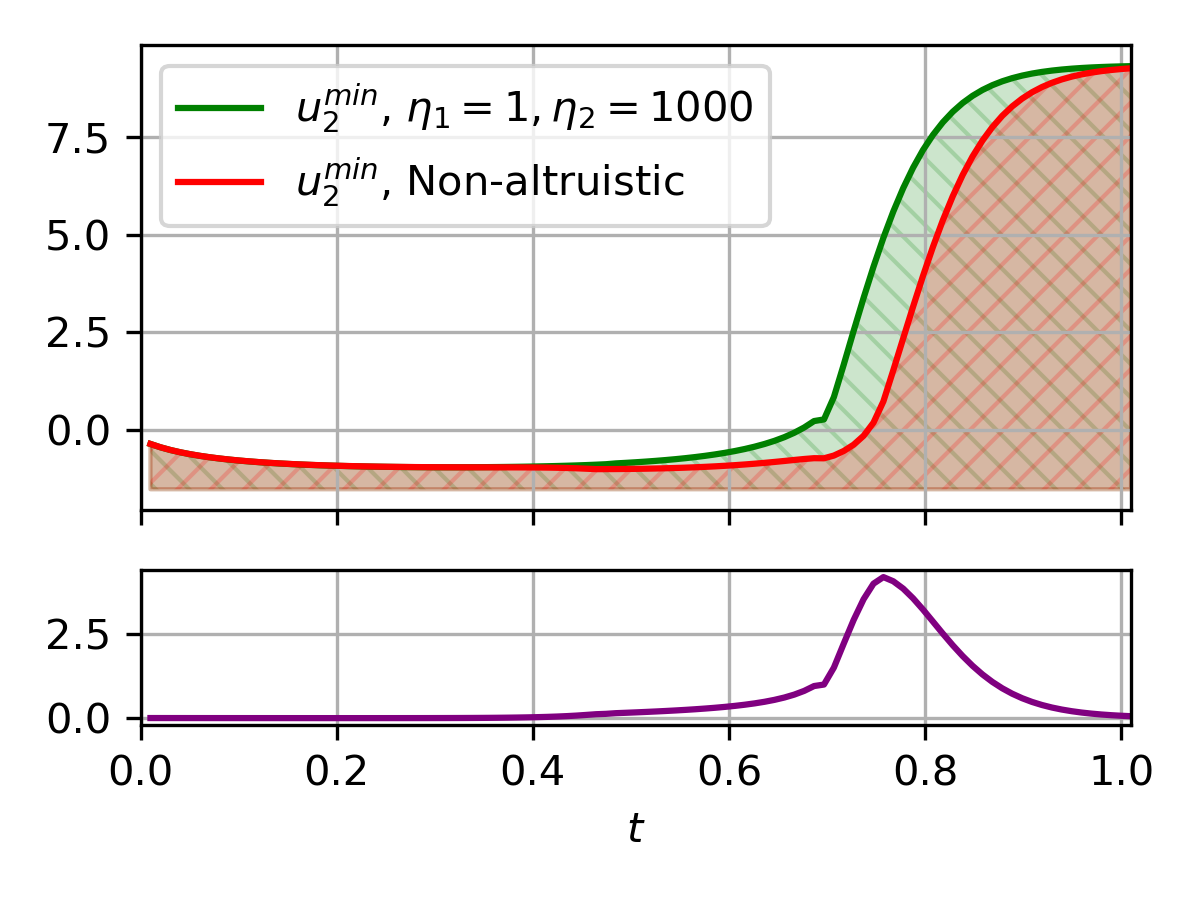}
        \caption{}
        \label{fig:sub:u_min_diff}
    \end{subfigure}
    \captionsetup{belowskip=-15pt}
    \caption{A simulation with the same initial conditions and parameterization as the simulation in Figure~\ref{fig:1D_example}, but with dynamics defined by \eqref{eq:2agent_1D_dyn_2inputs}. In (a), the control inputs $u_1, u_2$ are obtained by solving \eqref{eq:prob_statement_distr_alt} with auxiliary variable updates defined by \eqref{eq:y_updates}, where $\eta_1 = 1$ and $\eta_2=1000$. In (b), we compare the minimum $u_2$ that satisfies $u_2 \in \mathcal{U}_2^s(x_1, x_2, u_1)$ for two simulations: (red) $u_1, u_2$ are obtained by solving \eqref{eq:prob_statement_distr}, and (green) $u_1, u_2$ are obtained by solving \eqref{eq:prob_statement_distr_alt} with $\eta_1=1, \eta_2=1000$. We see that Theorem~\ref{thm:alt_feasibility} holds, with the altruistic system that is biased towards agent~2 yielding a larger feasible set of safe inputs while both agents remain safe as they approach the boundary. The difference between the two simulations is plotted beneath in purple.
    }
    \label{fig:1D_example_alt}
\end{figure}

\section{Conclusion} \label{sec:conclusion}
\vspace{-.5ex}
In this letter, we have introduced a framework for collaborative and altruistic safety in coupled multi-agent systems, leveraging ecologically inspired ideas from Hamilton’s rule. By extending collaborative control barrier functions with altruistic safety conditions, we showed how agents can trade off their own safety margins to prioritize more critical neighbors, thereby enlarging the feasible set of safe actions in distributed optimization. 
Future work will explore applying these ideas to more complex domains such as smart grid control and large-scale cyber-physical networks.

\vspace{-.5ex}

\normalem
\bibliographystyle{IEEEtran}
\bibliography{references}

\end{document}